\documentclass[12pt]{amsart} 
\usepackage{amsmath,amsfonts,amssymb,amsthm}
\usepackage[mathscr]{eucal}
\allowdisplaybreaks
\theoremstyle{plain}
\newtheorem{thm}{Theorem}[section]
\newtheorem{cor}[thm]{Corollary}
\newtheorem{lemma}[thm]{Lemma}

\theoremstyle{definition}

\newtheorem{remark}[thm]{Remark}

\numberwithin{equation}{section}

\def\Rn{{\mathbb R}^n}
\def\R{\mathbb R}
\def\R+{\mathbb R_+}
\def\M{\mathfrak M}

\def\O{\Omega_{0,1}}
\def\r{\rho}
\def\br{\overline{\rho}}

\def\gphq{\Gamma_{q,\phi_2}}

\def\rphd{\rho_{{}_{p,\phi}}^\prime}

\def\gphp{\Gamma_{p,\phi_1}}
\def\gph{\Gamma_{p,\phi}}

\def\rph{\rho_{{}_{p,\phi}}}
\def\rps{\rho_{{}_{p^\prime,\psi}}}

\def\lr{L_\rho}
\def\l{\lambda}
\def\i{\infty}

\def\ap{\approx}
\def\ls{\lesssim}

\def\ve{\varepsilon}
\def\bp{\overline{\phi}}
\begin{document}
\title[]{ The Rearrangement-Invariant space $\Gamma_{p,\phi}$ }
\author{Amiran Gogatishvili}
\address[Amiran Gogatishvili]{
Institute of Mathematics, Academy of Science of the Czech Republic, \v Zitn\'a 25, 11567 Prague 1, Czech Republic}
\email[]{gogatish@math.cas.cz}
%\urladdr{http://www.authorone.uni-aone.de}
\author{Ron Kerman}
\address[Ron Kerman]{Department of Mathematics, Brock University, 500 Glendridge Ave. St. Catharines, Ontario, Canada L2S 3A1}
\email[]{rkerman@brocku.ca}
%\urladdr{http://www.authortwo.uni-atwo.hu

\thanks{The research of the first author was partially supported by
the grant no. 201/08/0383 of the Grant Agency of the Czech Republic and  RVO: 67985840}
\thanks{The research of the  second  author was supported in part by NSERC grant A4021}
%\thanks{This paper is in final form and no version of it will be submitted for
%publication elsewhere.}
\date{}
%\dedicatory{Dedicated to Professor XY on the occasion of his seventieth birthday.}

\begin{abstract}Fix $b\in (0,\infty)$ and $p\in (1,\infty)$. Let $\phi$ be a positive measurable function on $I_b:=(0,b)$. Define the Lorentz Gamma norm, $\r_{p,\phi}$, at the measurable function $f:\R+\to\R+$ by $\rph(f):=\left[\int_0^bf^{**}(t)^p\phi(t)dt\right]^{\frac1p}$, in which $f^{**}(t):=t^{-1}\int_0^tf^{*}(s)ds$, where $f^*(t):=\mu_f^{-1}(t)$, with 
$\mu_f(s):=|\{ x\in I_b: |f(x)|>s\}|$. 

Our aim in this paper is to study the rearrangement-invariant space determined by $\rph$. In particular, we determine its K\"othe
dual and its Boyd indices. Using the latter a sufficient condition is given for a Cald\'eron-Zygmund operator to map such a space into itself.
\end{abstract}
\maketitle
\section{Introduction} \label{s1}

Let $(X,\mu)$ be a $\sigma$-finite measure space with $\mu(X)=b$ and denote by  $\M(X)$ the set of $\mu$-measurable real-valued functions on $X$.  
This paper is concerned with the properties of certain rearrangement invariant spaces of functions in $\M(X)$.
The norm of such a space is defined in terms of an index $p$, $1<p<\i$, and a positive locally integrable  (weight) function $\phi$ on $I_b:=(0,b)$ by 
\begin{equation} \label{0.1}
	\rph(f):=\left[\int_0^bf^{**}(t)^p\phi(t)dt\right]^{\frac1p}, \quad f\in \M(X).
\end{equation}
Here, 
$$f^{**}(t):=t^{-1}\int_0^bf^{*}(s)ds, \quad t\in I_b,
$$
in which the decreasing rearrangement, $f^*$, is the inverse (in a generalized sense) of the distribution function, 
$\mu_f$, of $f$, where 
$$\mu_f(\l):=\mu(\{ x\in X: |f(x)|>\l\}), \quad \l >0.
$$
We require 
$$ \int_1^\i \phi(t)t^{-p}dt<\i,  \,\, \text{if} \,\, b=\i, \,\, \text{and} \,\, \int_{I_b}\phi(t) t^{-p}dt=\i,\, \, \text{for all} \,\, b\in \R+;
$$
otherwise, the space 
$$
\gph=\gph(X):=\{f\in \M(X): \rph(f)<\i\}
$$
would, in the first case, consist only of the zero function and, in the second case, would be equal to the space $L_1(X)$ of $\mu$-integrable functions on $X$. Such weights $\phi$ will be called non-trivial.
 
 %With the norms $\ra$ in mind, we speak of 
 The {\it spaces} $\gph$  are examples of rearrangement-invariant (r.i) Banach function spaces, which are  defined by norms $\r$ whose characteristic property is that 
 $\r(f)=\r(g)$   whenever $ f,g\in \M(X)$ are equimeasurable in the sense that
 $f^*=g^*$.
 %here, $$f^*(t):=\inf\{\l>0: \mu(\{ x\in X: |f(x)|>\l\})\le t\},$$ $t\in I_b$. 
 
 A key thing to know about a Banach function norm, $\r$, such as \eqref{0.1}, is its associate norm, $\r^\prime$, defined at $g\in \M(X)$ by 
 $$\r^{\prime}(g)=\sup_{\substack{f\in \M(X)\\ \r(f)\le1}}\int_X |fg|d\mu.
 $$
 We will show that, when $ \gph(X)\not \supset L_\infty(X)$, or, equivalently, 
 $\int_{I_b}\phi(s)\,ds =\infty$, one has 
 $$ \rphd(g)\ap \rps(g), \quad g\in \M(X),$$
 where  $p^\prime=\frac{p}{p-1}$ and  $\psi$ is  a certain (dual) weight. 

We motivate the choice of $\psi$, in an appendix to the paper. For now, we just state our main result, namely,
\vskip+1.0cm 
 {\bf  Theorem A} {\it  Let $(X,\mu )$   be a $\sigma $-finite measure space with $\mu (X)=b$.
 Fix $p$, $1<p<\infty$, and suppose $\phi$ is a non-trivial weight function on $I _{b}$. Then,
\begin{equation*}\label{1.1}
\rphd(g)\approx \rps(g)+\frac{\int_{X}\left| g\right|}{\left[ \int_{I_{b}}\phi \right]^{p}}%+\frac{\esssup_{x\in X}\left| g(x)\right|}{\left[ \int_{I_{b}}\phi (t)t^{-p}dt\right]^{\frac{1}{p}}}
,\qquad 
g\in \M(X),
\end{equation*}
in which
 \begin{equation*}\label{1.2}
\psi (t):=\frac{t^{p+p^{\prime}-1}\int_{0}^{t}\phi \int _{t}^{b}\phi (s)s^{-p}ds}{\left[ \int _{0}^{t}\phi +t^{p}\int _{t}^{b}\phi (s)s^{-p}ds\right] ^{p^{\prime}+1}}, \qquad t\in I_{b}, \quad p'=\frac{p}{p-1}.
	\end{equation*}
}

 A proof of this theorem has been given 
 %This has been characterized for $\r=\rho _{\Gamma _{p,\phi}}$ 
 by the first   author and L. Pick in \cite{GP} using so-called discretization methods. Our aim here is to give a new proof using more familiar techniques. Alternative descriptions of the function space dual to $\gph$ can be found in \cite{GHS} and \cite{S1}.

The Boyd indices of  an r.i. norm  are essential to describing the action of such operators as those of Calder\'on-Zygmund on the space $\lr(\Rn)$. These indices are defined in terms of the norm, $h_\r(s)$, of the dilation operator. Their calculation when $\r=\rph$ and $\mu(X)=\i$ is  greatly simplified by the result in

\vskip+1.0cm 
 {\bf  Theorem B} {\it
 %Let $(X,\mu)$ be a $\sigma$-finite measure space with $\mu(X)=b$. 
Fix an index $p$, $1<p<\i$ and let $\phi$ be   a non-trivial  weight on $\R+$. Take $\r=\rph$ and at $s\in \R+$ set
$$ h_\r(s):=	\sup\frac{\r\left(f\left(\frac{t}{s}\right)\right)}{\r\left(f\right)}=
\sup\frac{\r\left(f^*\left(\frac{t}{s}\right)\right)}{\r\left(f^*\right)}, \quad 0\not =f\in \M_+(\R+).
$$
Then,
%\begin{equation*}\label{B1}
\[h_\r(s)\ap \sup_{t\in \R+}\left[\frac{\int_0^{st}\phi(y)dy +s^pt^p\int_{st}^{b} \phi(y)y^{-p}dy}{\int_0^{t}\phi(y)dy +t^p\int_{t}^{b} \phi(y)y^{-p}dy}\right]^{\frac1p}. 
\]
%\end{equation*}
}
  
\section{rearrangement-invariant spaces} \label{s2}
Let $(X,\mu)$ be a $\sigma$-finite measure space with $\mu(X)=b$ and denote by  $\M(X)$ the set of $\mu$-measurable real-valued functions on $X$ and by $\M_+(X)$ the nonnegative functions in   $\M(X)$.
A Banach function norm is a functional $\r: \M_+(X)\to \R+$ satisfying 
\begin{itemize}
	\item[(A1)] $\r(f)=0$ if and only if $f=0$ $\mu$- a.e.,
	\item[(A2)]  $\r(cf)=c\r(f)$, $c\ge 0$,
		\item[(A3)]  $\r(f+g)\le \r(f)+\r(g)$,
			\item[(A4)] $0\le f_n\uparrow f$ implies $\r(f_n)\uparrow\r(f)$,
			\item[(A5)] $|E|<\i$ implies $\r(\chi_E)<\i$,
     \item[(A6)] $|E|<\i$ implies $\int_Efd\mu\le c_E(\r)\r(f)$, for some constant $c_E(\r)$ depending  on $E$ and $\r$ but not on $f\in \M_+(X)$.
\end{itemize}

Furthermore, as mentioned in the introduction, a Banach function norm is said to be rearrangement invariant if $\r(f)=\r(g)$ whenever $f,g\in \M_+(X)$ are equimeasurable in the sense that 
 $f^*=g^*$. The decreasing rearrangement, $f^*$, of $f\in \M(X)$  on $\R+$ is defined as  
$$f^*(t):=\inf\{\l>0: \mu(\{ x\in X: |f(x)|>\l\})\le t\},$$
$t\in I_b$. It  satisfies the property that 
$$ |\{ t\in I_b: f^*(t)>\tau\})|=\mu(\{ x\in X: |f(x)|>\tau \}), \, \, f\in \M(X),
\,\,  \tau\in \R+.$$

Now, although the mapping $f\mapsto f^*$ is not subadditive, the mapping \newline
$f\mapsto t^{-1}\int_0^tf^*(s)ds $ is, namely ,

\begin{equation} \label{2.1}
t^{-1}\int_0^t(f+g)^*(s)ds \le t^{-1}\int_0^tf^*(s)ds+t^{-1}\int_0^tg^*(s)ds,
\end{equation}
 for all $f,g\in \M(X)$, $ t\in I_b.$
The Kothe dual of a Banach function norm $\r$ is another such norm, $\r'$, with
\begin{equation} \label{2.11}
\r'(g):=\sup_{\r(f)\le 1}\int_X fg\mu, \quad  f,g\in \M_+(X).
\end{equation}
It is obeys the Principle of Duality; that is, 
$$\r'':=(\r')'=\r.$$

The space $\lr(X)$ is the vector space 
$$\{f\in \M(X): \r(|f|)<\i\},$$
together with the norm 
$$\|f\|_{\lr}:=\r(|f|).$$
This Banach space is said to be an r.i. space provided $\r$ is an r.i. function norm.
The norm, $\rph$, defined in \eqref{0.1} in terms of an index $p$, $1<p<\i$, and a positive locally integrable  (weight) function $\phi$ on $I_b$ is an r.i. norm;

If  $\r$ is an  r.i. function norm, then,
\begin{equation}\label{ff}
\r(\chi_{{}_{(0,t)}})=\frac{t}{\r^\prime(\chi_{{}_{(0,t)}})},\quad t\in I_b
\end{equation}

The  dilation operator, $E_s$,  $s\in \R+$, given at  $f\in \M(\R+)$, $t\in \R+$, by
$$ (E_sf)(t):=f(st),$$
is bounded on any r.i. space $\lr(\R+)$ and  the operator norm of $E_{1/s}$ on  $\lr(\R+)$ is denoted
by  $h_\r(s)$. The norm is determined on the non-negative decreasing functions in $L_\r(\R+)$.

We define the lower and upper Boyd indices of $\lr(\R+)$ as
$$ 
i_\r:=\sup_{0<t<1}\frac{\log h_\r(t)}{\log t} \quad \text{and}\quad
I_\r:=\inf_{1<t<\i}\frac{\log h_\r(t)}{\log t}
$$

The operator norm of  $E_{1/s}$ on characteristic  functions of the form $ \chi_{{}_{(0,a)}}$, $a\in \R+$, is denoted by  $M_\r(s)$; thus,
$$
M_\r(s)=\sup_{0<a<\i}\frac{\r(\chi_{{}_{(0,as)}})}{\r(\chi_{{}_{(0,a)}})}.
$$
The so-called fundamental indices of $\r$ are defined in terms of $M_\r$ as

$$ 
\underline{i}_\r:=\sup_{0<s<1}\frac{\log M_\r(s)}{\log s} \quad \text{and}\quad
\underline{I}_\r:=\inf_{1<s<\i}\frac{\log M_\r(s)}{\log s}.
$$

Clearly, 
$$0\le i_\r\le \underline{i}_\r\le \underline{I}_\r\le I_\r
\le 1.$$

%%%%%%%%%%%%%%%%%%%%%
\section{Weighted spaces} \label{s3} Fix $b>0$ and let $w\in \M_+(I_b)$, $w>0$ a.e.. Given $p$, $1<p<\i$, the weighted Lebesgue space, $L_p(w)$, is defined by the norm
$$ \left[\int_0^b|f(t)|^pw(t)dt\right]^{\frac1p}, \quad f\in \M(I_b).$$
 One readily shows that the Banach dual of  $L_p(w)$ is the space $L_{p^\prime}(w^{1-p^\prime})$, $p^\prime=\frac{p}{p-1}$, namely, the weighted Lebesgue space with norm 
  $$ \left[\int_0^b|g(t)|^{p^\prime}w(t)^{1-p^\prime}dt\right]^{\frac1{p^\prime}}, \quad g\in \M(I_b).$$
 
 In this section we consider the action of certain positive integral operators  on such spaces. This action  is expressed on terms of so-called weighted norm inequalities. The most basic ones involve the Hardy averaging operator and its dual, that is, 
 
 \[(Pf)(t):=t^{-1}\int_0^t f(s)ds\quad \text{and} \quad  (Qf)(t):=\int_t^b f(s)\frac{ds}{s}, \quad f\in \M_+(I_b),\quad t\in I_b. 
\]
 \begin{thm}[\cite{M}]\label{t3.1}
 Fix $b>0$  and let $u$ and $v$ be weights on $I_b$. Then, for $1<p\le q<\i$ one has the least constant $C>0$ in the inequality
 \begin{equation} \label{h3.1}
\left(\int_0^b (u(t)(Pf)(t))^q\,dt\right)^{\frac1q}\le C\left(\int_0^b(v(t) f(t))^p\,dt\right)^{\frac1p}, \quad f\in \M_+(I_b),
\end{equation}
  equivalent to
  \[\sup_{0<r<b} \left(\int_r^b \left(\frac{u(t)}{t}\right)^q\,dt\right)^{\frac1q}\left(\int_0^r v(t)^{-p^\prime}\,dt\right)^{\frac1{p^\prime}},
  \]
  and the least constant $C>0$ in the inequality
  \begin{equation} \label{q3.2}
\left(\int_0^b (u(t)(Qf)(t))^q\,dt\right)^{\frac1q}\le C\left(\int_0^b(v(t) f(t))^p\,dt\right)^{\frac1p}, \quad f\in \M_+(I_b),
\end{equation}
  equivalent to
  \[\sup_{0<r<b} \left(\int_0^r u(t)^q\,dt\right)^{\frac1q}\left(\int_r^b (tv(t))^{-p^\prime}\,dt\right)^{\frac1{p^\prime}}.
 \]
  \end{thm}
 An operator essentialy built from $P$ and $Q$ when $b=\i$ is the Stieltjes operator
  $$(Sf)(t):=\int_0^\i\frac{f(s)}{s+t}ds,\quad f\in \M_+(I_b).$$
    Clearly, for $f\in \M_+(\R+), \quad t\in \R+,$
  \begin{equation}\label{3.31}\frac12[(Pf)(t)+(Qf)(t)]\le (Sf)(t)\le [(Pf)(t)+(Qf)(t)].% \quad f\in \M_+(\R+),\quad t\in \R+. 
  \end{equation}
  
The following results are given in Andersen~\cite{A} for $1<p\le q<\i$ and in Sinnamon~\cite{S} for $1<q<p<\i$.
  \begin{thm}\label{t3.2}
 Let $u$ and $v$ be weights on $\R+$. Then,  in the inequality 
 \begin{equation} \label{s3.1}
\left(\int_0^\i (Sf)(t)^qu(t)\,dt\right)^{\frac1q}\le K\left(\int_0^\i f(t)^pv(t)\,dt\right)^{\frac1p}, \quad f\in \M_+(\R+),
\end{equation}
the least possible  $K>0$ is  
  equivalent to
   \begin{equation} \label{a3.1}\sup_{t>0} \left(\int_0^\i \left(\frac{t}{s+t}\right)^qu(s)\,ds\right)^{\frac1q}\left(\int_0^\i \frac{v(t)^{1-p^\prime}}{(s+t)^{p^\prime}}\,ds\right)^{\frac1{p^\prime}},
  \end{equation}
  when $1<p\le q<\i$,
  and to
  \[\left[\int_0^\i\left[\left[\int_0^\i \left(\frac{t}{s+t}\right)^qu(s)\,ds\right]^{\frac1p}\left[\int_0^\i \frac{v(t)^{1-p^\prime}}{(s+t)^{p^\prime}}\,ds\right]^{\frac1{p^\prime}}\right]^{\frac{pq}{p-q}}u(t)\,dt\right]^{\frac1q-\frac1{p}},
 \]
  when $1<q<p<\i$.
  \end{thm}
 
\section{Proof of Theorem A.}\label{s4}
The following lemma is a key element in the proof of the Theorem A. In it and  in the rest of the  section, it will simplify  things if we write $\psi$  in the form
\begin{equation}\label{3.1}
	\psi =\frac{(P\phi)(Q_{p}\phi)}{\left[(PQ_{p})(\phi)\right] ^{p^\prime+1}},
	\end{equation}
where 
%\begin{equation}\label{M4}
\[(P\phi)(t)=t^{-1}\int_0^t \phi(s)ds\quad \text{and} \quad  (Q_{p}\phi)(t)=pt^{p-1}\int_t^b \phi(s)s^{-p}ds. 
\]%\end{equation}

\begin{lemma}\label{l3.1}
Fix $p$ and $b$ with, $1<p<\infty$ and $0<b\leq \infty$. Suppose $\phi$ is a non-trivial weight on $I_{b}$ and let $\psi$ be given by \eqref{3.1} . Then, there exists $C>0$, independent of  $ f,g\in \M_+(I_{b})$, such that 

\[{\rm (i)}\quad \int _{I_{b}}fg\left[ \frac{P\phi }{(PQ_{p})(\phi)}\right] ^{\frac{1}{p^\prime}+1} \leq C\left( \int _{I_{b}}f^{p}\phi \right) ^{\frac{1}{p}}\left[ \left( \int _{I_{b}}(Pg)^{p^\prime}\psi  \right)^{\frac{1}{p^\prime}} +\frac{\int _{I_{b}}g}{\left[ \int _{I_{b}}\phi \right] ^{\frac{1}{p}}}\right], 
\]
if $f\downarrow$,
and
\[{\rm (ii)} \quad
\int _{I_{b}}fg\left[ \frac{Q_{p}\phi }{(PQ_{p})(\phi)}\right] ^{\frac{1}{p^\prime}+1}\leq C\left( \int _{I_{b}}f(t)^{p}\phi (t)t^{-p}dt\right) ^{\frac{1}{p}}
	 \left( \int _{I_{b}}\left( \int _{t}^{b}g\right) ^{p^\prime}\psi (t)dt \right) ^{\frac{1}{p^\prime}},
	 \]
if $f\uparrow$.
 \end{lemma}
 \begin{proof}
 \,\, 
  
 {\rm (i)} We have 
 \begin{align*}
\frac{p^\prime}{p^\prime+1}&\int _{I_{b}}fg\left[ \frac{P\phi }{\left(PQ_{p}\right)(\phi)}\right]  ^{\frac{1}{p^\prime}+1}\\%	&=\int _{I_{b}}f(t)g(t)\int _{0}^{t}\left( \int_{0}^{s}\phi \right) ^{\frac{1}{p^\prime}}\phi (s)ds\left[ t((PQ_{p})\phi )(t)\right]^{-\frac{1}{p^\prime}-1}dt\\
&=\int_{I_{b}}g(t)\int_{0}^{t}f(s)\left( \int_{0}^{s}\phi \right)^{\frac{1}{p^\prime}}\phi (s)ds\left[ t\left(PQ_{p}\right)(\phi) (t)\right] ^{-\frac{1}{p^\prime}-1}dt , \quad \text{since} \quad f\downarrow,\\
&=\int_{I_{b}}f(t)\left( \int_{0}^{t}\phi \right) ^{\frac{1}{p^\prime}}\int_{t}^{b}g(s)\left[ s\left( PQ_{p}\right)(\phi) (s)\right] ^{-\frac{1}{p^\prime}-1}ds\phi (t)dt,\quad \text{by Fubini's theorem},\\
&=\int_{I_{b}}f(t)\left( \int_{0}^{t}\phi \right) ^{\frac{1}{p^\prime}}
\left[\left.\int_{0}^{s}g\left[ s\left(PQ_{p}\right)(\phi) (s)\right]^{-\frac{1}{p^\prime}-1}\right|_t^b\right.\\
&\hskip+1.5cm+\left.\left(\frac{1}{p^\prime}+1\right) \int_{t}^{b}\int_{0}^{s}g
\left[  s\left(PQ_{p}\right)(\phi) (s)\right]^{-\frac{1}{p^\prime}-2}
 \left( Q_{p}\phi \right)(s)ds\right] \phi (t)dt \\
&\leq \int_{I_{b}}f(t)\left( \int_{0}^{t}\phi \right)^{\frac{1}{p^\prime}}\left[ \int_{I_{b}}g\left[ \int_{I_{b}}\left(Q_{p}\phi\right) (t)dt \right] ^{-\frac{1}{p^\prime}-1}\right.\\
&\hskip+0.5cm \left.+\left( \frac{1}{p^\prime}+1\right) \int_{t}^{b}\int_{0}^{s}g\left[ s\left( PQ_{p} \right)(\phi)  (s)\right] ^{-\frac{1}{p^\prime}-2}\left( Q_{p}\phi \right) (s)ds\right] \phi (t)dt\\ 
&=\int_{I_{b}}f(t)\left( \int_{0}^{t}\phi \right)^{\frac{1}{p^\prime}}\left[ p^{\frac{1}{p^\prime}+1}\int_{I_{b}}g\left[ \int_{I_{b}}\phi \right] ^{-\frac{1}{p^\prime}-1}\right.\\
&\hskip+0.5cm \left.+\left( \frac{1}{p^\prime}+1\right)  \int_{t}^{b}\int_{0}^{s}g\left[ s\left(PQ_{p}\right)(\phi)  (s)\right]^{-\frac{1}{p^\prime}-2}\left( Q_{p}\phi \right) (s)ds\right] \phi (t)dt\\ % {\rm max}((p^{\frac{1}{p^\prime}+1}, \left( \frac{1}{p^\prime}+1\right))
&\leq (p+1)^{2}\left[ \int_{I_{b}}f^{p}\phi \right] ^{\frac{1}{p}}
\left[ \int_{I_{b}}\int_{0}^{t}\phi  \left[ \int_{I_{b}}g\left[ \int_{I_{b}}\phi \right] ^{-\frac{1}{p^\prime}-1}\right.\right.\\
&\hskip+0.5cm +\left.\left.\int_{t}^{b}\int_{0}^{s}g\left[s\left(PQ_{p}\right)(\phi)  (s)\right] ^{-\frac{1}{p^\prime}-2}\left( Q_{p}\phi \right) (s)ds\right] ^{p^\prime}\phi (t)dt\right] ^{\frac{1}{p^\prime}}\\
&\leq (p+1)^{2}\left[ \int_{I_{b}}f^{p}\phi \right] ^{\frac{1}{p}}
   \left[\int_{I_{b}}g\left[\int_{I_{b}}\phi \right]^{-\frac{1}{p^\prime}-1}\left[ \int_{I_{b}}\phi (t)\int_{0}^{t}\phi dt\right] ^{\frac{1}{p^\prime}}\right.\\
 &\hskip+0.5cm +\left.\left[ \int_{I_{b}}\left( \int_{t}^{b}\int_{0}^{s}g\left[ s\left(PQ_{p}\right)(\phi) (s)\right] ^{-\frac{1}{p^\prime}-2}\left( Q_{p}\phi\right) (s)ds\right) ^{p^\prime}\phi (t)\int_{0}^{t}\phi dt\right] ^{\frac{1}{p^\prime}}\right]\\
&\leq (p+1)^{2}\left[ \int_{I_{b}}f^{p}\phi \right] ^{\frac{1}{p}}
 \left[  \int_{I_{b}}g\left[\int_{I_{b}}\phi \right]^{-\frac{1}{p}}\right.\\
 &\hskip+0.5cm +\left.\left[ \int_{I_{b}}\left( \int_{t}^{b}\int_{0}^{s}g\left[ s\left(PQ_{p}\right)(\phi) (s)\right] ^{-\frac{1}{p^\prime}-2}( Q_{p}\phi) (s)ds\right) ^{p^\prime}\phi (t)\int_{0}^{t}\phi dt\right] ^{\frac{1}{p^\prime}}\right] ,
 \end{align*}
 in which the third inequality was obtained using H\"older's inequality with respect to the measure $\phi (t)dt$. 
 
 The proof of (i) will be complete if we can show, that  
 \[ %\begin{align}\label{3.2}
	\int_{I_{b}}\left( \int_{t}^{b}\int_{0}^{s}g\left[ s\left( PQ_{p} \right)(\phi )(s)\right] ^{-\frac{1}{p^\prime}-2}\left(Q_p\phi \right)(s)ds\right) ^{p^\prime}\phi (t)\int_{0}^{t}\phi dt
	\] %\\  &\leq C\int_{I_{b}}(Pg)(t)^{p^\prime}\psi (t)dt. \notag \end{align}
is dominated by a constant multiple of $\int_{I_{b}}(Pg)(t)^{p^\prime}\psi (t)dt$.
To this end, let
\[
	H(t):=\int_{0}^{s}g\left[ s\left( PQ_{p}\right)(\phi )(s)\right] ^{-\frac{1}{p^\prime}-2}\left( Q_{p}\phi \right)(s), \qquad s\in I_{b}
\]
so that the assertion reads
\begin{align*}
	\int_{I_{b}}\left( \int_{t}^{b}H(s)ds\right) ^{p^\prime}&\phi (t)\int_{0}^{t}\phi dt\\
	&\leq C\int_{I_{b}}H(t)^{p^\prime}\left[ t\left( PQ_{p}\right)(\phi) (t)\right] ^{p^\prime}\left(\left( Q_{p}\phi \right) (t)\right)^{-p^\prime-1}\int_{0}^{t}\phi (t)dt
\end{align*}
But, this holds by Theorem~\ref{t3.1}, % \cite[(1.7) and (1.8) on page 12]{KP},
 inasmuch as
\begin{align*}
&\left( \int_{0}^{t}\phi (s)\int_{0}^{s}\phi ds\right) ^{\frac{1}{p^\prime}}\left( \int_{t}^{b}\left[ s\left( PQ_{p}\right) (\phi)(s)\right] ^{-p}\left( Q_{p}\phi \right) (s)\left( \int_{0}^{s}\phi \right)^{1-p}ds\right) ^{\frac{1}{p}} \\
&=2^{-\frac{1}{p^\prime}}\left( \int_{0}^{t}\phi \right) ^{\frac{2}{p^\prime}}\left( \int_{t}^{b}\left[ s\left( PQ_{p}\right)(\phi) (s)\right] ^{-p}\left( Q_{p}\phi \right) (s)\left( \int_{0}^{s}\phi \right)^{1-p}ds\right) ^{\frac{1}{p}}\\
&\le 2^{-\frac{1}{p^\prime}}\left( \int_{0}^{t}\phi \right) ^{\frac{2}{p^\prime}}\left( \int_{0}^{t}\phi \right) ^{-\frac{1}{p^\prime}}\left( \int_{t}^{b}\left[ s\left( PQ_{p}\right)(\phi)(s)\right] ^{-p}d\left[ s\left( PQ_{p}\right)(\phi) (s)\right]\right) ^{\frac{1}{p}}\\
&=-\frac{2^{-\frac{1}{p^\prime}}}{p-1}\left( \int_{0}^{t}\phi \right) ^{\frac{1}{p^\prime}}\left(\left.\left[ s\left( PQ_{p}\right)(\phi)(s)\right] ^{-p+1}\quad\right|_{t}^{b}\right)^{\frac{1}{p}}\\
&\le \frac{2^{-\frac{1}{p^\prime}}}{p-1}\left[ \frac{\int_{0}^{t}\phi }{t\left( PQ_{p}\right)(\phi)(t)}\right] ^{\frac{1}{p^\prime}}\\
&=\frac{2^{-\frac{1}{p^\prime}}}{p-1}\left[ \frac{\int_{0}^{t}\phi }{\int_{0}^{t}\phi +t\left( Q_{p}\phi \right) (t)}\right] ^{\frac{1}{p^\prime}}\\
&\leq \frac{2^{-\frac{1}{p^\prime}}}{p-1}.
\end{align*}
{\rm (ii)}
To begin, suppose $b=\i$. Making the change of variable $t\rightarrow t^{-1}$ three times in a row and setting $\widetilde{f}(y)=f(y^{-1})$, $\widetilde{g}(y)=g(y^{-1})$, $\widetilde{\phi}(y)=\phi (y^{-1})y^{p-2}$, we obtain
\begin{align*}
\int_{0}^{\i }fg\left[ \frac{Q_{p}\phi }{\left( PQ_{p}\right) (\phi) }\right] ^{\frac{1}{p^\prime}+1}&=\int_{0}^{\i }\widetilde{f}\widetilde{g}\left[ \frac{\int_{t^{-1}}^{\i }\phi (s)s^{-p}ds}{t^{p}\int_{0}^{t^{-1}} Q_{p}\phi }\right] ^{\frac{1}{p^\prime}+1}dt,\\
&=\int_{0}^{\i }\widetilde{f}\widetilde{g}\left[ \frac{\int_{0}^{t}\widetilde{\phi} }{t^{p}\int_{t}^{\i }s^{-p-1}\int_{0}^{s}\widetilde{\phi}ds }\right] ^{\frac{1}{p^\prime}+1}dt \\
&=\int_{0}^{\i }\widetilde{f}\widetilde{g}\left[ \frac{\int_{0}^{t}\widetilde{\phi} }{t\left( Q_{p}P\right)(\widetilde{\phi}) (t) }\right] ^{\frac{1}{p^\prime}+1}dt\\
&=\int_{0}^{\i }\widetilde{f}\widetilde{g}\left[ \frac{P\widetilde{\phi} }{\left(PQ_{p}\right)(\widetilde{\phi})}\right] ^{\frac{1}{p^\prime}+1}.
\end{align*}
Thus, from (i), there follows, since $\widetilde{f}\downarrow$,
\[
\int_{0}^{\i }fg\left[ \frac{Q_{p}\phi }{\left( PQ_{p}\right)( \phi )}\right] ^{\frac{1}{p^\prime}+1}\leq C\left[ \int_{0}^{\i }\widetilde{f}^{p}\widetilde{\phi }\right] ^{\frac{1}{p}} \left[ \left( \int_{0}^{\i }(P\widetilde{g})^{p^\prime}\widetilde{\psi }\right) ^{\frac{1}{p^\prime}} +\frac{\int_{0}^{\i }\widetilde{g}}{\left[ \int_{0}^{\i }\widetilde{\phi }\right] ^{\frac{1}{p}}}\right],
\]
with
\[
\widetilde{\psi }=\frac{(P\widetilde{\phi })(Q_{p}\widetilde{\phi })}{\left[ (PQ_{p})(\widetilde{\phi}) \right] ^{p^\prime+1}}.
\]
Now, the change of variable $t\rightarrow t^{-1}$ yields
\begin{align*}
\int_{0}^{\i }\widetilde{f}(t)\widetilde{\phi }(t)dt&=\int_{0}^{\i }\widetilde{f}(t^{-1})^{p}\widetilde{\phi }(t^{-1})t^{-2}dt\\
&=\int_{0}^{\i }f(t)^{p}\phi (t)t^{-p}dt,\\
\int_{0}^{\i }\widetilde{g}(t)dt &=\int_{0}^{\i }\widetilde{g}(t^{-1})t^{-2}dt=\int_{0}^{\i }g(t)dt\\
\end{align*}
and
\begin{align*}
\int_{0}^{\i }\widetilde{\phi }(t)dt &=\int_{0}^{\i }\widetilde{\phi }(t^{-1})t^{-2}dt=\int_{0}^{\i }\phi (t)t^{-p}dt=\infty
\end{align*}
Again,
\begin{align*}
\widetilde{\psi }(t) &=\frac{t^{-1}\int_{0}^{t}\widetilde{\phi }(s)dst^{p-1}\int_{t}^{\i }\widetilde{\phi}(s)s^{-p}ds}
{\left[ \left( PQ_{p}\right)(\widetilde{\phi })(t) \right] ^{p^\prime+1}}\\
&=\frac{t^{-1}\int_{0}^{t}\phi (s^{-1})s^{p-2}dst^{p-1}\int_{t}^{\i }\phi (s^{-1})s^{-2}ds}
{\left[ \frac{t^{-1}}{p}\int_{0}^{t}\phi (s^{-1})s^{p-2}ds+\frac{t^{p-1}}{p}\int_{t}^{\i }\phi (s^{-1})s^{-2}ds\right] ^{p^\prime+1}}\\
&=\frac{t^{-1}\int_{0}^{t^{-1}}\phi (s)dst^{p-1}\int_{t^{-1}}^{\i }\phi (s)s^{-p}ds}
{\left[\frac{t^{-1}}{p}\int_{t^{-1}}^{\i }\phi (s)s^{-p}ds+\frac{t^{p-1}}{p}\int_{0}^{t^{-1}}\phi (s)ds \right] ^{p^\prime+1}}\\
&=t^{p^\prime-2}\frac{(P\phi )(t^{-1})(Q_{p}\phi )(t^{-1})}{\left[\left( PQ_{p}\right)(\phi) (t^{-1})\right]^{p^\prime+1}}\\
&=\psi (t^{-1})t^{p^\prime-2}.
\end{align*}
So,
\begin{align*}
\int_{0}^{\i }(P\widetilde{g})(t)^{p}\widetilde{\psi }(t)dt &=\int_{0}^{\i }\left( t^{-1}\int_{0}^{t}g(s^{-1})s^{-2}ds\right) ^{p^\prime}\psi (t^{-1})t^{p^\prime-2}dt\\
&=\int_{0}^{\i }\left( \int_{0}^{t}g(s^{-1})s^{-2}ds\right) ^{p^\prime}\psi (t^{-1})t^{-2}dt\\
&=\int_{0}^{\i }\left( \int_{t}^{\i }g(s)ds\right) ^{p^\prime}\psi (t)dt.
\end{align*}
This completes the proof of (ii) when $b=\i $. In the case $b<\i $, a similar argument works if we replace the transformation $t\rightarrow t^{-1}$ by $t\rightarrow (b-t)^{-1}$.
\end{proof}

{\bf Proof of Theorem A.} 
We first show
\begin{equation}\label{3.3}
	\rph^{\prime}(g)\geq c\left[ \rps(g)+\frac{\int_{I_{b}}|g|}{\left[ \int_{I_{b}}\phi \right] ^{\frac{1}{p}}}%+\frac{\esssup _{t\in I_b}\left| g(t)\right| }{\left[ \int_{I_{b}}\phi (t)t^{-p}dt
	%\right] ^{\frac{1}{p}}}
	\right],
\end{equation}
for some $c>0$ independent of $g\in \M_+ (I_b)$. To this end, it suffices, in view of \eqref{2.11}, to find constants $C$, $c>0$, independent of $g\in \rph^{\prime}$, to which there corresponds an $f\in \M_+(I_b)$, with $f\downarrow $, $\rph(f)\leq C$ and
\begin{equation}\label{3.4}
	\int_{I_{b}}fg^{*}\geq c\left[ \rps(g^{*})+\frac{\int _{I_{b}}g^{*}}{\left[ \int_{I_{b}}\phi \right] ^{\frac{1}{p}}}%+\frac{g^{*}(0+)}{\left[ \int_{I_{b}}\phi (t)t^{-p}dt\right] ^{\frac{1}{p}}}
	\right].
\end{equation}
Fixing $g$, we seek
$$ f= Qh$$ 
for some  $h$ in $\M_+(I_b)$.

We need a condition on $h$ to guarantee $\rph(Qh)<\i $.
But,% with $(Sh)(t)=\int_0^b\frac{h(s)}{s+t}\,ds, \quad t\in I_b$, one has 
\begin{align*}
\rph(Qh) &=\left[ \int_{I_{b}}\left( (PQ)h\right)^{p}\phi\right] ^{\frac{1}{p }}\\
&=\left[ \int_{I_{b}}\left( Ph+Qh\right)^{p}\phi\right] ^{\frac{1}{p}}\\
&\leq 2\left[ \int_{I_{b}} (Sh)^{p}\phi\right] ^{\frac{1}{p}}\\
&\leq B\left[ \int_{I_{b}}h^{p}\psi ^{1-p}\right] ^{\frac{1}{p}},
\end{align*}
the last inequality being  proved in  the Appendix. The  desired condition on $h$ is thus
\[
\int_{I_{b}}h^{p}\psi ^{1-p}<\i .
\]
As pointed out in Section~\ref{s3}, the weighted Lebesgue norms  
\[
\left[ \int_{I_b}g^{p^\prime}\psi \right] ^{\frac{1}{p'}} \quad \text{and} \quad \left[ \int_{I_b}h^p\psi ^{1-p} \right] ^{\frac{1}{p}} ,\qquad  g,h\in \M_+(I_{b}),
\]
are dual to one another. Therefore, for our given $g\in L_{\rph^{\prime}}$, there exists $ h_{0}\in \M_+ (I_{b})$, such that
\[
\int_{I_{b}}h_{0}^{p}\psi ^{1-p}\leq 1
\]
and
\[
\int_{I_{b}}g^{*}Qh_{0}=\int_{I_{b}}h_{0}Pg^{*}\geq \frac12\left[ \int_{I_{b}}(g^{**})^{p^\prime}\psi\right] ^{\frac{1}{p^\prime}}=\frac12\rho_{{}_{p^\prime,\psi }}(g).
\]
If $\int_{I_{b}}\phi <\i $, the constant function with value 
$$\frac{1}{\left[ \int_{I_{b}}\phi \right] ^{\frac{1}{p}}}$$
 will belong to $\gph$ with norm $1$ and
\begin{equation}\label{3.5}
f:= Qh_{o}+\frac{1}{\left[ \int_{I_{b}}\phi \right] ^{\frac{1}{p }}}
\end{equation}
will satisfy 
\[
\int_{I_{b}}fg^{*}\geq \frac{\int_{I_{b}}g^{*}}{\left[ \int_{I_{b}}\phi \right] ^{\frac{1}{p }}}.
\]
%If $\int_{I_{b}}\phi (t)t^{-p}dt<\i $ and
%\[
%g^{**}(t_{0})\geq \frac{g^{**}(0+)}{2}=\frac{g^{*}(0+)}{2}, \quad t_{0}\in I_{b},
%\]
%when
%\begin{equation}\label{3.5}
%f:=Qh_{0}+\frac{1}{\left[ \int_{I_{b}}\phi \right] ^{\frac{1}{p}}}%+\frac{t_{0}^{-1}\chi_{(0,t_{0})}}{\left[ \int_{I_{b}}\phi (t)t^{-p}dt\right]^{\frac{1}{p}}},
%\end{equation}
%we have 
%\[
%\int_{I_{b}}fg^{*}\geq \frac{t_{0}^{-1}\int_{0}^{t_{0}}g^{*}}{\left[ \int_{I_{b}}\phi (t)t^{-p}dt\right]^{\frac{1}{p}}}\geq \frac12\frac{g^{*}(0+)}{\left[ \int_{I_{b}}\phi (t)t^{-p}dt\right]^{\frac{1}{p}}} .
%\]
%we observe that
%\[
%\rph\left( \frac{t_{0}^{-1}\chi(0,t_{0})}{\left[ \int_{I_{b}}\phi (t)t^{-p}dt\right]^{\frac{1}{p}}} \right) =\frac{t}{\left[ \int_{I_{b}}\phi (t)t^{-p}\right]^{\frac{1}{p}}}\left[ \int_{0}^{t_{0}}t_{0}^{-p}\phi (t)dt\right] ^{\frac{1}{p}}\leq 1 .
%\]
Altogether, then, the function $f$ defined in \eqref{3.5} has $\rph(f)\leq C=B+1$ and satisfies \eqref{3.4} with $c=\frac12$. 

We now prove the inequality opposite to \eqref{3.3}, this being equivalent to 
	
\begin{equation}\label{3.6}
	\int_{I_{b}}f^{*}g^{*}\leq C\rph(f^{*})\left[ \rps(g^{*})+\frac{\int_{I_{b}}g^{*}}{\left[ \int_{I_{b}}\phi \right] ^{\frac{1}{p}}}%+\frac{g^{*}(0+)}{\left[ \int_{I_{b}}\phi (t)t^{-1}dt\right] ^{\frac{1}{p}}}
	\right] ,
\end{equation}
in which $C>0$ is independent of $f,g\in \M (X)$.

It suffices to consider $g^{*}$ of the form 
\[
g^{*}=k+Qh ,\quad k\geq 0 \quad \text{and} \quad h \in \M_+ (I_{b}).
\]
For the term 
\[
\frac{\int_{I_{b}}g^{*}}{\left[ \int_{I_{b}}\phi \right] ^{\frac{1}{p}}}
\]
to be finite we require $b=\mu (X)<\i $ or $\int_{I_{b}}\phi =\i $. In either case, the term is dominated by an absolute constant times $\rps(g^{*})$ and is irrelevant.
%The only time a positive $k$ is of interest, then, is when $\int_{I_{b}}\phi (t)t^{-p}dt<\i $. This means
%$$\gph\subset L_1$$
%or, what amounts to the same thing,
%$$L_\i\subset \gph^\prime .$$
%So,
%\begin{align*}
%\int_{I_{b}}f^{*}k&\le\rph(f^*) \rph^\prime(k)\\
%&\le B \rph(f^*) \cdot k,
%\end{align*}
%which implies \eqref{3.6} for $g^*=k$.

We have only to to consider those $g^*$ of the form $g^*=Qh$, $ h\in \M_+(I_b)$. For such $g^*$, 
\begin{equation}\label{3.7}
\begin{array}{ll} \int_{I_{b}}f^{*}g^*&=\int_{I_{b}}f^{*}Qh= \int_{I_{b}}hP(f^*)=\int_{I_{b}}f^{**}h\\
&=\int_{I_{b}}f^{**}h\left[(PQ_p)(\phi)\right]^{\frac{1}{p^\prime}+1}\left[(PQ_p)(\phi)\right]^{-\frac{1}{p^\prime}-1}\\
&=p^{-\frac{1}{p^\prime}-1}\int_{I_{b}}f^{**}h \left[\frac{P\phi+Q_p\phi}{(PQ_p)(\phi)}\right]^{\frac{1}{p^\prime}+1}\\
&\le\left(\frac{2}{p}\right)^{\frac{1}{p^\prime}+1}
\left[\int_{I_{b}}f^{**}h \left[\frac{P\phi}{(PQ_p)(\phi)}\right]^{\frac{1}{p^\prime}+1}
+\int_{I_{b}}f^{**}h \left[\frac{Q_p\phi}{(PQ_p)(\phi)}\right]^{\frac{1}{p^\prime}+1}\right]\\
&=\left(\frac{2}{p}\right)^{\frac{1}{p^\prime}+1}\left[I_1+ I_2\right].
\end{array}
\end{equation}

Since $f^{**}\downarrow$,  Lemma~\ref{l3.1}, (i), gives 
\begin{equation}\label{3.8}
 I_1\le C\rph(f)\left[ \left(\int_{I_{b}}\left(Ph\right)^{p^\prime}\psi\right)^{\frac{1}{p^\prime}}+
\frac{\int_{I_{b}}h}{\left[\int_{I_{b}}\phi\right]^{\frac{1}{p}}}\right].
\end{equation}
But,
$$g^{**}=Pg^{*} = (PQ)h= Ph+Qh\ge Ph$$
and
 $$ \int_{I_{b}}g^{*} =\int_{I_{b}}Qh=\int_{I_{b}}h,$$
whence \eqref{3.8} implies 
\begin{equation}\label{3.9}
I_1\le C\rph(f)\left[\rps(g^*)+\frac{\int_{I_{b}}g^{*}}{\left[\int_{I_{b}}\phi\right]^{\frac{1}{p}}}\right].
\end{equation}
  
  Observing that $\int_0^tf^* \uparrow$, Lemma~\ref{l3.1}, {\rm (ii)}, ensures
  
\begin{equation}\label{3.10}
I_2 \le C\left[ \int_{I_{b}}\left(\int_0^tf^*\right)^p \phi(t)t^{-p}dt\right]^{\frac{1}{p}}
%\left[
\left(\int_{I_{b}}\left(Qh\right)^{p^\prime}\psi\right)^{\frac{1}{p^\prime}}%+
%\frac{\int_{I_{b}}\frac{h(t)}{t}dt}{\left[\int_{I_{b}}\phi(t) t^{-p}dt\right]^{\frac{1}{p}}}
%\right]\\
= C\rph(f^*)%\left[
\rps(g^*)%+\frac{g^{*}(0+)}{\left[\int_{I_{b}}\phi(t)t^{-p}dt\right]^{\frac{1}{p}}}\right].
\end{equation}
  
  Combining \eqref{3.7}, \eqref{3.9} and \eqref{3.10} yields \eqref{3.6} and thereby completes the proof.  $ \square$
 
 \begin{cor} \label{phi} Let $\phi$  be a non-trivial weight function on $\R+$, and $\psi$ its dual weight. Then, 
 \begin{equation}\label{phps}
t^{-p'}\left[\int_0^t \psi(s)\,ds	+t^{p'}\int _t^\infty \psi(s)s^{-p'}\,ds\right]\ap \left[\int_0^t\phi(s)\,ds+ t^p\int_t^\infty \phi(s)s^{-p}\,ds\right]^{1-p'}.
\end{equation}
 \end{cor}
 \begin{proof} It is easy to see that
 \[\rph(\chi_{{}_{(0,t)}})=\left(\int_0^t\phi(s)\,ds+ t^p\int_t^\infty \phi(s)s^{-p}\,ds\right)^{\frac1{p}}\]
 and 
 \[\rps(\chi_{{}_{(0,t)}})=\left(\int_0^t \psi(s)\,ds	+t^{p'}\int _t^\infty \psi(s)s^{-p'}\,ds\right)^{\frac1{p^\prime}}.\]
 %Now, use the equality \eqref{ff}  and the fact that $\rph$ and $\rps$ are r.i. function norms and $\rps$ is associate norm of  $\rph$ we get \eqref{phps}.
    Since $\rph$  and $\rps$  are  associate  r.i. function norms, 
    \eqref{phps} now follows from \eqref{ff}. 
    \end{proof}
 
 \begin{cor}\label{co} Fix $p\in (1,\i)$ and  suppose $\phi$ is a  non-trivial weight function on $\R+$, with 
 \[%\begin{equation}\label{O1}
\int_0^\infty\phi(t)dt=\i. %	\int_0^\i t^{-p}\phi(t)dt=
\]%\end{equation}
 Then,% there exists  $C>0$,  independent of $g\in \M_+(\R+)$, such that 
 \begin{equation}\label{O2}
 %C^{-1}\left(\int_0^\i(Sg)^{p^\prime}\psi\right)^{\frac1{p^\prime}}\le 
 \sup_{f\in\O(\R+)}\frac{\int_0^\i fg}{\left(\int_0^\i f^p\phi\right)^{\frac1p}}\ap  \left(\int_0^\i(Sg)^{p^\prime}\psi\right)^{\frac1{p^\prime}}, \quad g\in \M_+(\R+), 
	\end{equation}
 in which $\psi$ is the weight  dual to $\phi$ and 
 $$\O(\R+):=\{f\in \M_+(\R+):\,\, tf(t)\uparrow \,\, \text{and} \,\, f \downarrow\}$$
 \end{cor}
 \begin{proof}
As pointed out in  \cite[p. 117]{BL}, $f\in \O(\R+)$ if and only if 
 $$ \frac12t^{-1}\int_0^t h^{*}(s)ds\le f(t)\le 2 t^{-1}\int_0^t h^{*}(s)ds,$$
  for some $h\in \M_+(\R+)$.
 Hence, the left side of \eqref{O2}, is equivalent to 
 \begin{align*}
 \sup_{h\in \M_+(\R+)}\frac{\int_0^\i t^{-1}\int_0^t h^{*}(s)dsg(t)dt}{\rph(h)}
 &=\sup_{h\in \M_+(\R+)}\frac{\int_0^\i h^{*}(t)\int_t^\i g(s)\frac{ds}{s}\,dt}{\rph(h)}\\
 &=\rph^\prime\left(\int_t^\i g(s)\frac{ds}{s}\right)\\
 &\ap\rps\left(\int_t^\i g(s)\frac{ds}{s}\right),
 \end{align*}
 which yields \eqref{O2}, in view of \eqref{3.31}, since 
  \begin{align*} \rps\left(\int_t^\i g(s)\frac{ds}{s}\right)
  &= \left(\int_0^\i \left(t^{-1}\int_0^t\int_s^\i g(y)\frac{dy}{y}ds\right)^{p^\prime}\psi(t)dt\right)^{\frac1{p^\prime}}\\
 &\ap \left(\int_0^\i(Sg)^{p^\prime}\psi\right)^{\frac1{p^\prime}}.
   \end{align*}
    \end{proof}

 \begin{thm}\label{tR} Fix $p,q\in (1,\i)$.  Suppose $\phi_1$ and $\phi_2$ are  weights on $\R+$, with  $\phi_1$ and  its dual weight $\psi_1$ as in Corollary~\ref{co}. Let $T$ be a positive linear operator on $\M_+(\R+)$
 with associate operator $T^\prime$. Then,
% \begin{equation} \label{R1}
%	\r_{{}_{q,\phi_2}}(Tf)\ls\r_{{}_{p,\phi_1}}(f), \quad f\in \O(\R+),
%\end{equation}
\begin{equation} \label{R1}
\left(\int_0^\i (Tf)^q\phi_2\right)^{\frac1q}\le C\left(\int_0^\i f^p\phi_1\right)^{\frac1p}, \quad f\in \O(\R+),
\end{equation}
  if and only if
 % \begin{equation} \label{R2}
 % \r_{{}_{p^\prime,\psi_1}}(ST^\prime h)\ls\r_{{}_{q^\prime,\psi_2^{1-q^\prime}}}(h), %\quad h\in \M(\R+),
 %	\end{equation}
\begin{equation} \label{R2}
  \left(\int_0^\i(ST^\prime)(h)^{p^\prime}\psi_1\right)^{\frac1{p^\prime}}\le K
  \left(\int_0^\i h^{q^\prime}\phi_2^{1-q^\prime}\right)^{\frac1{q^\prime}}, \quad h\in \M_+(\R+),
 	\end{equation}
 	 	or
\begin{equation} \label{R21}
  \left(\int_0^\i(TS)(h)^q\phi_2\right)^{\frac1q}\le K
 \left(\int_0^\i h^{p}\psi_1^{1-p}\right)^{\frac1{p}}, \quad h\in \M_+(\R+).
 	\end{equation}
	% provided $\int_0^\infty\phi_1(t)dt=\i$. 
	  Here, $K\ap C$.
	   \end{thm}
 \begin{proof}
  The reverse H\"older inequality ensures that 
 \eqref{R1} is equivalent to 
 \[
\frac{\int_0^\i (Tf)h}{\left(\int_0^\i h^{q^\prime}\phi_2^{1-q^\prime}\right)^{\frac1{q^\prime}}}\ls 
\left(\int_0^\i f^p\phi_1\right)^{\frac1p}, \quad f\in \O(\R+), \,\, h\in \M_+(\R+),
	\]
 or 
  \begin{equation}\label{R3}
\frac{\int_0^\i f(T^\prime h)}{\left(\int_0^\i f^p\phi_1\right)^{\frac1p}}\ls \left(\int_0^\i h^{q^\prime}\phi_2^{1-q^\prime}\right)^{\frac1{q^\prime}}, \quad f\in \O(\R+), \,\, h\in \M_+(\R+).
	\end{equation}
 In view of Corollary~\ref{co}, \eqref{R3} amounts to
 $$
 \r_{{}_{p^\prime,\psi_1}}((ST^\prime) (h))\ls
 \left(\int_0^\i h^{q^\prime}\phi_2^{1-q^\prime}\right)^{\frac1{q^\prime}}, \quad h\in \M_+(\R+),
 $$
 that is, \eqref{R2}. 
As we have 
$$\int_0^\i (ST^\prime)h(t) g(t)dt = \int_0^\i h(t)(TS)g(t)dt, \quad h, g\in \M_+(\R+),$$
\eqref{R2} is equivalent to \eqref{R21}, by the  duality theorem for
 weighted Lebesgue spaces . 
\end{proof} 
  
 \section{Imbeddings and Boyd indices} \label{s5}
 
 \begin{thm}\label{tI}
 Fix $p,q\in (1,\i)$.  Suppose $\phi_1$ and $\phi_2$ are  weights  on $\R+$, with $\phi_1$ and its dual weight $\psi_1$ as in Corollary~\ref{co}. Then, the (possibly infinite) norm of 
 the imbedding
 \begin{equation}\label{I1}
 \gphp(\R+)\hookrightarrow\gphq(\R+)
\end{equation}	
is equivalent to 
\begin{equation}\label{I2}
		\sup_{t>0}\frac{\left[\int_0^t\phi_2(s)\,ds+ t^q\int_t^\infty \phi_2(s)s^{-q}\,ds\right]^{\frac1{q}}}{\left[\int_0^t\phi_1(s)\,ds+ t^p\int_t^\infty \phi_1(s)s^{-p}\,ds\right]^{\frac1{p}}}, 
\end{equation}
if $1<p\le q<\i$, and to
\begin{equation}\label{I3}
\left[\int_0^\i\left[\frac{\int_0^t\phi_2(s)\,ds+ t^q\int_t^\infty \phi_2(s)s^{-q}\,ds}{\int_0^t\phi_1(s)\,ds+ t^p\int_t^\infty \phi_1(s)s^{-p}\,ds}\right]^{\frac{q}{p-q}}\phi_2(t)dt\right]^{\frac1{q}-\frac 1p},
 \end{equation}
 %\begin{equation}\label{I31}
%\int_0^\i\left(\r_{{}_{q,\phi_2}}(\chi_{{}_{(0,t)}})\right)^{r}\frac{(P\phi_1)(t)(Q_{p}\phi_1)(t)}{\left[(PQ_{p})(\phi)_1\right] ^{\frac{r}{q}+1}}dt<\i 
%\end{equation}
 if $1<q< p <\i$.%, and $r=\frac{pq}{p-q}$.
  \end{thm}
\begin{proof} The imbedding  \eqref{I1} is equivalent  to an inequality of the form 
\begin{equation} \label{I4}
	\r_{{}_{q,\phi_2}}(If)\le C\r_{{}_{p,\phi_1}}(f), \quad f\in \O(\R+),
\end{equation}
or 
\begin{equation} \label{I61}
  \left(\int_0^\i (If)^q\phi_2\right)^{\frac1q}\le C  \left(\int_0^\i f^p\phi_1\right)^{\frac1p}, \quad f\in\O(\R+),
 	\end{equation}
 	in which $I$ is the identity operator. According to Theorem~\ref{tR}, \eqref{I61} reduces to 
\begin{equation} \label{I6}
  \left(\int_0^\i(S h)^q\phi_2\right)^{\frac1q}\le K  \left(\int_0^\i h^p\psi_1^{1-p}\right)^{\frac1p}, \quad h\in \M_+(\R+);
 	\end{equation}
 	here, $K\ap C$ and  
 	\[ \psi_1(t) =\frac{(P\phi_1)(t)(Q_{p}\phi_1)(t)}{\left[(PQ_{p})(\phi_1)\right] ^{p^\prime+1}}.
 	 	\] 
 By Theorem~\ref{t3.2}, the least possible $K$ in \eqref{I6} is equivalent to 
 \begin{equation}\label{I7}
\sup_{t>0}\left[\int_0^\i \frac{\psi_1(s)}{(s+t)^{p^\prime}}ds\right]^{\frac{1}{p^\prime}}
	\left[\int_0^\i \left(\frac{t}{s+t}\right)^{q}\phi_2(s)ds\right]^{\frac{1}{q}}
\end{equation}
 	when $1<p\le q<\i$, and to %while Sinnamon \cite{S} shows the least possible $K$  in \eqref{I6} is equivalent to 
 	\begin{equation}\label{I8}
\left[\int_0^\i \left[\int_0^\i \frac{\psi_1(s)}{(s+t)^{p^\prime}}ds\right]^{\frac{(p-1)q}{p-q}}	
\left[\int_0^\i \left(\frac{t}{s+t}\right)^{q}\phi_2(s)ds\right]^{\frac{q}{p-q}}\phi_2(t)dt\right]^{\frac1q-\frac1p}, 
\end{equation}
 	when $1<q<p<\i$.
  	But,
 	$$\int_0^\i \left(\frac{t}{s+t}\right)^{q}\phi_2(s)ds\ap \int_0^t\phi_2(s)ds +t^q\int_t^\i \phi_2(s)s^{-q}ds
 %	\ap \rgphq(\chi_{{}_{(0,t)}})^q,
 	$$ 
 	and 
 \begin{align*} \int_0^\i \frac{\psi_1(s)}{(s+t)^{p^\prime}}ds
 	 	&\ap t^{-p^\prime}\int_0^t\psi_1(s)ds +\int_t^\i \psi_1(s)s^{-p^\prime}ds\\
 	 	& =t^{-p^\prime}\left[ \int_0^t\psi_1(s)ds +t^{p^\prime} \int_t^\i \psi_1(s)s^{-p^\prime}ds  \right]\\
 	 	&\ap\left[ \int_0^t\phi_1(s)ds +t^{p} \int_t^\i\phi_1(s) s^{-p}ds  \right]^{1-p^\prime},
 	 	\end{align*}
 	 	by Corollary~\ref{phi}, so \eqref{I7}  becomes \eqref{I2} and  \eqref{I8} becomes \eqref{I3}.
\end{proof}
\begin{thm}\label{tB} 
%Let $(X,\mu)$ be a $\sigma$-finite measure space with $\mu(X)=\i$. 
Fix an index $p$, $1<p<\i$ and suppose $\phi$ is  a non-trivial  weight on $\R+$. Take $\r=\rph$ on $\M_+(\R+)$. Then,
\begin{equation}\label{B1}h_\r(t)\ap M_\r(t)\ap\sup_{s\in \R+}\left[\frac{\int_0^{st}\phi(y)dy +s^pt^p\int_{st}^{\infty} \phi(y)y^{-p}dy}{\int_0^{s}\phi(y)dy +s^p\int_{s}^{\infty} \phi(y)y^{-p}dy}\right]^{\frac1p}, \quad t\in \R+,
\end{equation}
and
\begin{equation}\label{B2} i_\r= \underline{i}_\r, \quad  I_\r =\underline{I}_\r.
\end{equation}
\end{thm}
\begin{proof} For $f\in \M_+(\R+)$, $f$ decreasing, we have 
$$\left(E_{\frac1t}f\right)^{**}(s)=f^{**}\left(\frac st\right), \quad s\in \R+,$$
so 
$$\rph\left(E_{\frac1t}f\right)=\r_{{}_{p,\overline{\phi}}}(f),$$
where 
$$\overline{\phi}(s)=t\phi(st).$$
Thus, for $t\in \R+$,
\begin{align*}
	h_\r(t)&=\sup_{f\in \M_+(\R+) \atop 
	f\downarrow}\frac{\rph\left(E_{\frac1t}f\right)}{\rph\left(f\right)}\\
&=\sup_{f\in \M_+(\R+)}\frac{\r_{{}_{p,\overline{\phi}}}(f)}{\rph\left(f\right)}\\
&\ap \sup_{s\in \R+}
\frac{\r_{{}_{p,\overline{\phi}}}\left(\chi_{{}_{(0,s)}}\right)}
{\rph\left(\chi_{{}_{(0,s)}}\right)}, \quad \text{by Theorem \ref{tI}},\\
&\ap\sup_{s\in \R+}
\left[\frac{\int_0^{s}t\phi(ty)dy +s^p\int_{s}^{\i} t\phi(ty)y^{-p}dy}{\int_0^{s}\phi(y)dy +s^p\int_{s}^{\i} \phi(y)y^{-p}dy}\right]^{\frac1p}\\
&\ap\sup_{s\in \R+}\left[\frac{\int_0^{st}\phi(y)dy +s^pt^p\int_{st}^{\i} \phi(y)y^{-p}dy}{\int_0^{s}\phi(y)dy +s^p\int_{s}^{\i} \phi(y)y^{-p}dy}\right]^{\frac1p}\\
&\ap \sup_{s\in \R+}\frac{\rph(\chi_{{}_{(0,st)}})}{\rph(\chi_{{}_{(0,s)}})}\\
&=M_\r(t).
\end{align*}
%The equality in \eqref{B2} follows from equivalency \eqref{B1}, just using the definitions of indecies.
\end{proof}

\begin{remark} The formula \eqref{B1}, now proved, is the one asserted in Theorem~B.
\end{remark}

%%%%%%%%%%%%%%%%%%%%%%%%%%%%%%%%%%%%%%%%

\section{Calder\'on-Zygmund Operators} \label{s6}

A function $K$, on $\Rn \setminus\{0\}$,  locally integrable away from the
origin, is said to be a Calder\'on-Zygmund (CZ) kernel, 
provided it satisfies the following four conditions:

{\rm (i)} There exists a constant $C_1>0$, independent  of $\ve$ and $N$, $ 0 <\ve < N$, 
such that 
\[\left| \int_{\ve<|x|<N}K(x)dx\right|\le C_1;
\]
moreover, for each $N>0$, one has the existence of
\[ \lim_{\ve\to0+}\int_{\ve<|x|<N}K(x)dx.
\]

{\rm(ii)} There exists a constant $C_2>0$, independent  of $R>0$, for which
\[\int_{|x|<R}|x||K(x)|dx\le C_2R.
\]
{\rm (iii)} There exists a constant $C_3>0$, independent  of $y\in \Rn\setminus\{0\}$, with
\[\int_{|x|>2|y|}|K(x-y)-K(x)|dx \le C_3.
\]
{\rm (iv)} There exists a constant $C_4>0$, independent  of $R>0$ and of points $x_1$, $x_2$ and $x_3$ in $\Rn$ within  a distance $\frac R2$ of one another and each a distance 
 %$x_1$, $x_2$ and $x_3$ within  a distant $\frac R2$ of each other and each a distant
greater then $R$ from $y$, such that 
\[|K(x_1-y)-K(x_2-y)|\le C_4\frac{|x_1-x_2|}{|x_3 -y|^{n+1}}.
\]

The Calder\'on-Zygmund operator, $T_K$, with kernel $K$, is the singular integral operator 

\[(T_Kf)(x):=\lim_{\ve\to 0+}\int_{|x-y|>\ve}K(x-y)f(y)dy,\quad x\in \Rn,
\] 
which is defined a.e. for all $f\in \M(\Rn)$ with 
\[\int_{\Rn}\frac{|f(y)|}{1+|y|^n}dy<\i.
\] 

\

\begin{thm}\label{TCZSI} Fix $p$, $1<p<\i$, and suppose the weight $\phi$ on $\R+$  satisfies
\[\int_0^\i \phi(s)\min\left[1,s^{-p}\right]\,ds<\i \quad\text{and}\quad \int_0^\i\phi(s)\max\left[1, s^{-p}\right]ds=\i.
\]  

Denote by $\psi$ the function defined in \eqref{3.1}.

Let $T_K$ be a CZ operator. Then, one has 
\begin{equation}\label{CZ1}
T_K: \gph(\Rn)\to \gph(\Rn)
\end{equation}
 if there  exists $c$, $0<c<1$, such that for all $t\in \R+$, 

\begin{equation}\label{CZ}
\begin{array}{ll}
	\int_{0}^{ct}\phi(s)ds+ c^pt^p \int_{t}^{\i}\phi(s)s^{-p}ds
	&\le\frac12
\left[\int_{0}^{t}\phi(s)ds+ t^p \int_{t}^{\i}\phi(s)s^{-p}ds\right]\\
\\
\int_{0}^{ct}\psi(s)ds+ c^{p^\prime}t^{p^\prime} \int_{t}^{\i}\psi(s)s^{-{p^\prime}}ds
&\le\frac12
\left[\int_{0}^{t}\psi(s)ds+ t^{p^\prime} \int_{t}^{\i}\psi(s)s^{-{p^\prime}}ds\right]. 
 \end{array} 
 \end{equation}
\end{thm}
\begin{proof}

Let $\r$ be an r.i. norm on  $\M_+(\R+)$ defined in terms of r.i. norm $\br$ on $\M_+(\R+)$ by $\r(f)=\br(f^*)$. It is shown in \cite{K} that 
 \[
T_K: L_\r(\Rn)\to L_\r(\Rn)
\]

provided 
\[\lim_{s\to0+}sh(s)=0=\lim_{s\to \i}h(s),
\]
where $h(s)=h_{\br}(\frac1s)$. In terms of $h_{\br}(s)$ and  $h_{\br^\prime}(s)=sh(s)$, these 
conditions read 
\begin{equation}\label{limh}
	\lim_{s\to0+}h_{\br}(s)=0=\lim_{s\to0+}h_{\br^\prime}(s).
\end{equation}
The inequalities
\[h_{\br}(s_1s_2)\le h_{\br}(s_1)h_{\br}(s_2) \quad \text{and} \quad h_{\br^\prime}(s_1s_2)\le h_{\br^\prime}(s_1)h_{\br^\prime}(s_2), \quad s_1,s_2\in \R+,
\]
imply that, given $\varepsilon>0$, \eqref{limh} is equivalent to the existence of $c$, $0<c<1$, for which $h_\r(c)<\varepsilon$ and $h_{\r^\prime}(c)<\varepsilon$.
By Theorem~B, then,  \eqref{limh} is equivalent to \eqref{CZ}, when $\br=\rph$.
\end{proof}

\begin{remark}%Remark 6.1  
The condition \eqref{CZ} is also necessary for \eqref{CZ1}  
when, for example, $T_K$ is the Hilbert transform or one of the Riesz  
transforms.
\end{remark}

\

\section{Appendix}\label{s7}
It is our purpose here to give an heuristic argument to motivate the choice of $\psi$ in \eqref{3.1} when $\phi$ is a non-trivial weight on $I_b$ satisfying $\int_{I_b}\phi=\i$.

Now,
 $$\rphd(g)=\sup_{f\in \M_+(I_b)}
 \frac{\int_0^b f^*(t)g^*(t)dt}{\left[\int_0^bf^{**}(t)^p\phi(t)dt\right]^{\frac1p}}=:I(g), \quad g\in\M_+(I_b).
 $$
 It suffices to consider $f(t)=\int_t^b h(s)\frac{ds}{s}$ for some $h\in \M_+(I_b)$, $h\not = 0$  a.e..
 Since, in that case, 
  $$\int_0^b f^*(t)g^*(t)=\int_0^b\int_t^bh(s)\frac{ds}{s}g^*(t)dt=\int_0^b h(t)g^{**}(t)dt$$
  and
  $$f^{**}(t)=t^{-1}\int_0^t\int_s^bh(y)\frac{dy}{y}ds=t^{-1}\int_0^t h(s)ds+\int_t^bh(s)\frac{ds}{s}\ap (Sh)(t), \quad t\in I_b,$$
 % where
 % $$(Sh)(t)=\int_0^b\frac{h(s)}{s+t}ds,$$
 % is the Stieltjes operator $S$ at $h$.
    we have 
  $$I(g)=\sup_{h\in \M_+(I_b)}
 \frac{\int_0^b h(t)g^{**}(t)dt}{\left[\int_0^b (Sh)(t)^p\phi(t)dt\right]^{\frac1p}}.
 $$
 
 If $\bp$ is such that 
  
\begin{equation} \label{M1}
\int_0^b (Sh)^p\phi\le C\int_0^b h^p\bp, \quad h\in \M(I_b),
\end{equation}
then, 
\[ I(g)\ge C^{-1}\sup_{h\in\M_+(I_b)}\frac{\int_0^bh(t)g^{**}(t)\,dt}{\left[ \int_0^bh(t)^p\bp(t)\,dt\right]^{\frac1p}}.
\]

This suggests we take $\psi(t)=\bp(t)^{1-p^\prime}$ where $\bp$ is, in some sense the smallest weight such that \eqref{M1} holds. Andersen's  condition \eqref{a3.1} for \eqref{M1}
leads us to solve for $ \bp(t)^{1-p^\prime}$ in the equation
\begin{equation} \label{M2}
\left[\int_0^b \frac{\phi(s)}{(s+t)^{p}}ds\right]^{\frac{1}{p}}
	\left[\int_0^b \left(\frac{t}{s+t}\right)^{p^\prime}\bp(s)^{1-p^\prime}ds\right]^{\frac{1}{p^\prime}}=1,
\end{equation}
or, what is equivalent,
\[\int_0^t \bp(s)^{1-p^\prime}ds+t^{p^\prime}\int_t^b \bp(s)^{1-p^\prime}s^{-p^\prime}ds
=\left[t^{-p}\int_0^t \phi(s)ds+ \int_t^b \phi(s)s^{-p}ds\right]^{1-p^\prime}.
\]
Differentiation with respect to $t$ yields 
\[t^{p^\prime-1}\int_t^b \bp(s)^{1-p^\prime}s^{-p^\prime}ds
= \int_0^t\phi(s)ds
\left[t^{-p}\int_0^t \phi(s)ds+ \int_t^b \phi(s)s^{-p}ds\right]^{-p^\prime}.
\]
Differentiating again with respect to $t$ we get

\[ \bp(t)^{1-p^\prime}=
\frac{pp^\prime t^{pp^\prime-1}\int_0^t \phi(s)ds \int_t^b \phi(s)s^{-p}ds}{\left[\int_0^t \phi(s)ds+t^{p} \int_t^b \phi(s)s^{-p}ds\right]^{p^\prime+1}}
-\frac{t^{p^\prime}\phi(t)}{\left[\int_0^t \phi(s)ds+t^{p} \int_t^b \phi(s)s^{-p}ds\right]^{p^\prime}}
\]
It seems we essentially have 
\begin{equation}\label{M3}
\bp(t)^{1-p^\prime}=\frac{(P\phi)(t)(Q_{p}\phi)(t)}{\left[(P\phi)(t)+(Q_{p}\phi)(t)\right] ^{p^\prime+1}}. 
\end{equation}
%where 
%\begin{equation}\label{M4}
%(P\phi)(t)=t^{-1}\int_0^t \phi(s)ds\quad \text{and} \quad  (Q_{p}\phi)(t)=pt^{p-1}\int_t^b s^{-p}\phi(s)ds. 
%\end{equation}

 The weight $\widehat{\phi}(t)$ given by 
\[\widehat{\phi}^{1-p^\prime}(t)= \frac{t^{p^\prime}\phi(t)}{\left[\int_0^t \phi(s)ds+t^{p} \int_t^b \phi(s)s^{-p}ds\right]^{p^\prime}}
= \frac{\phi(t)}{\left[(P\phi)(t)+(Q_{p}\phi)(t)\right]^{p^\prime}}
\] 
is readily shown to satisfy Andersen's condition \eqref{M2} and, hence, so will 
\[ \frac{(P\phi)(t)(Q_{p}\phi)(t)}{\left[(P\phi)(t)+(Q_{p}\phi)(t)\right] ^{p^\prime+1}}=
\bp(t)^{1-p^\prime}+\widehat{\phi}(t)^{1-p^\prime}.
\]
Now, $\bp(t)$ will be better then $\widehat{\phi}(t)$ in \eqref{M1} if 
\[\int_0^b g^{**}(t)^{p^\prime}\widehat{\phi}(t)^{1-p^\prime}dt \le C
 \int_0^b g^{**}(t)^{p^\prime} \bp(t)^{1-p^\prime}dt.
 \]
  One readily infers from Theorem~\ref{tI} that this will be so if and only if
  \[\int_0^t s^{p^\prime-1}\int_s^b \widehat{\phi}(y)^{1-p^\prime}y^{-p^\prime}dy ds\le C
 \int_0^t s^{p^\prime-1}\int_s^b  \bp(y)^{1-p^\prime}y^{-p^\prime}dyds.
 \]
 But,
  \begin{align*}
  \int_s^b  \bp(y)^{1-p^\prime}y^{-p^\prime}dy
  &\ap\int_s^b y^{-p^\prime}
  \frac{(P\phi)(y)(Q_{p}\phi)(y)}{\left[(P\phi)(y)+(Q_{p}\phi)(y)\right]^{p^\prime+1}}dy\\
  &=\int_s^b
  \frac{y^{p-1}\int_0^y \phi(z)dz \int_y^b \phi(z)z^{-p}dz}
  {\left[\int_0^y \phi(z)dz+y^{p} \int_y^b \phi(z)z^{-p}dz\right]^{p^\prime+1}}dy\\
  &=-\frac1{p^\prime}\int_s^b \int_0^y \phi(z)dz 
  \frac{d}{dz}\left[\int_0^y \phi(z)dz+y^{p} \int_y^b\phi(z) z^{-p}dz\right]^{-p^\prime}\,dy\\
  &=-\frac1{p^\prime} \int_0^y \phi(z)dz 
\left[\int_0^y \phi(z)dz+y^{p} \int_y^b\phi(z) z^{-p}dz\right]^{-p^\prime}\Bigg |_s^b\\
&\hskip+1cm +\frac1{p^\prime}
\int_s^b \phi(y)\left[\int_0^y \phi(z)dz+y^{p} \int_y^b \phi(z)z^{-p}dz\right]^{-p^\prime}dy\\
   &\ge \frac1{p^\prime}\int_s^b \widehat{\phi}(y)^{1-p^\prime}y^{-p^\prime}dy,
  \end{align*}
  if $\int_0^b\phi(z)dz=\i$.

\end{document}